\documentclass[draft]{article}
\usepackage[dvips]{epsfig}
\input{epsf}
\usepackage{amsmath,amssymb,amsfonts}
\usepackage{amscd}
\usepackage[dvips,matrix,arrow]{xy}

\newtheorem{theorem}{Theorem}%[section]

\newtheorem{example}[theorem]{Example}

\newenvironment{proof}[1][Proof:]{\begin{trivlist}
\item[\hskip \labelsep {\bfseries #1}]}{\end{trivlist}}

\newcommand{\qed}{\nobreak \ifvmode \relax \else
      \ifdim\lastskip<1.5em \hskip-\lastskip
      \hskip1.5em plus0em minus0.5em \fi \nobreak
      \vrule height0.75em width0.5em depth0.25em\fi}

 0 3 \font\ccc =msbm10 \font\cccc = msbm7

\begin{document}

\title{A note on certain finite-dimensional representations of the braid group}

\author{Valentin Vankov Iliev\\
Section of Algebra,\\ Institute of
Mathematics and Informatics,\\ Bulgarian Academy of Sciences, 1113
Sofia, Bulgaria\\
e-mail: viliev@math.bas.bg}

\maketitle

\begin{abstract}

In this paper the author finds explicitly all finite-dimensional 
irreducible representations of a series of finite permutation groups 
that are homomorphic images of Artin braid group. 

\end{abstract}


\section{Introducton}

\label{I}

In \cite{[1]} we consider certain permutation representations
of Artin braid group $B_n$ on $n$ strands in
appropriate finite symmetric groups. The corresponding 
permutation groups $B_n(\sigma)$ depend on a permutation $\sigma$ 
and are extensions of the symmetric group $S_n$ on $n$ letters by an
abelian group $A_n(\sigma)$. Let $q$ be the order of a permutation
$\tau$ which is bound up with $\sigma$, see \cite[(1)]{[1]}. If
$q$ is odd, then any one of these extensions splits, and the
groups $A_n(\sigma)$ and $B_n(\sigma)$ depend only on $q$:
$B_n(q)=B_n(\sigma)$,  $A_n(q)=A_n(\sigma)$. Moreover, in this
case the monodromy action of $S_n$ on $A_n(q)$ is very simple.
This allows a straightforward application of the method of "little
groups" of Wigner and Mackey for finding all finite-dimensional
irreducible representations of $B_n(q)$, which is done in the
present paper. By composing with the surjective homomorphism
$B_n\to B_n(q)$ we obtain a series of finite-dimensional
irreducible representations of $B_n$.

\section{The theorem}

\label{IV}

This paper is a natural continuation of \cite{[1]}, and we use freely the terminology
and notation, introduced there.  Throughout the end we assume that the
order $q$ of the permutation $\tau\in S_d$ from \cite[(1)]{[1]} is
odd. In accord with \cite[Corollary 10]{[1]}, up to an isomorphism
of abstract groups, the groups $A_n(\sigma)$ and $B_n(\sigma)$
depend only on the positive integer $q$ and we adapt the notation:
$B_n(q)=B_n(\sigma)$, and $A_n(q)=A_n(\sigma)$. Now, making use of
\cite[Theorem 9, (iii)]{[1]}, we obtain that $B_n(q)$ is
isomorphic to the semidirect product of the symmetric group $S_n$
by the abelian group
\[
A_n(q)=\underbrace{\hbox{\ccc Z}/(q)\coprod\cdots\coprod
\hbox{\ccc Z}/(q)}_{\hbox{$n$ times}},
\]
and $A_n(q)$ has a structure of $S_n$-module, given by the
monodromy homomorphism $m$ from \cite[Proposition 11]{[1]}. The
group $A_n(q)$ can be
considered as a free $\hbox{\ccc Z}/(q)$-module of rank $n$ with
basis $e_1=(1,0,\ldots, 0)$, $e_2=(0,1,\ldots, 0)$, $\ldots$,
$e_n=(0,0,\ldots, 1)$. We identify $A_n(q)$ with its 
multiplicatively written version $\langle\tau\rangle^{\left(n\right)}$ via the rule $x_1e_1+\cdots+x_ne_n\mapsto \tau^{x_1}\omega^{d}(\tau^{x_2})\ldots\omega^{\left(n-1\right)d}
(\tau^{x_n})$. The involution $m(\theta_s)=\iota_s$, $s=1,\ldots,
n-1$, acts on the basis $(e_i)$ by the rule $\iota_se_s=e_{s+1}$,
and $\iota_se_r=e_r$ for $r\neq s,s+1$. Therefore, if
$\zeta\in\Sigma_n$, we have
\[
\zeta\cdot x=(x_{\zeta^{-1}\left(1\right)},\ldots, x_{\zeta^{-1}\left(n\right)})
\]
for any $x=x_1e_1+\cdots+x_ne_n\in A_n(q)$. The contragredient
action of the symmetric group $\Sigma_n$ on the dual $\hbox{\ccc Z}/(q)$-module
$A_n(q)^*$ is given by the formulae
\[
\zeta\cdot b=(b_{\zeta^{-1}\left(1\right)},\ldots, b_{\zeta^{-1}\left(n\right)})
\]
for any $b=b_1x_1+\cdots+b_nx_n\in A_n(q)^*$.

We denote by $P_n$ the set of all partitions of $n$. Any linear
form $b=(b_1,\hdots, b_n)\in A_n(q)^*$ defines a family of
non-negative number $(\ell_k^{\left(b\right)})_{k\in \hbox{\cccc
Z}/\left(q\right)}$, where  $\ell_k^{\left(b\right)}=|\{i\in
[1,n]\mid b_i=k\}|$. This family, after ordering from largest to
smallest, produces a partition
$\lambda^{\left(b\right)}=(\lambda_1^{\left(b\right)},\lambda_2^{\left(b\right)},\hdots
)\in P_n$. Let $P_{\leq q;n}$ be the set of all partitions
$\lambda\in P_n$ with length not exceeding $q$. The map
$A_n(q)^*\to P_n$, $b\mapsto\lambda^{\left(b\right)}$, is
$\Sigma_n$-equivariant, and let $t\colon\Sigma_n\backslash
A_n(q)^*\to P_n$ be its factorization.

Let $\hbox{\ccc C}$ be the field of complex numbers, and let
$\hbox{\ccc C}^*$ be the multiplicative group of its non-zero
elements. Let us fix a primitive $q$th root of unity
$\varepsilon\in\hbox{\ccc C}^*$. Let $X=Hom(A_n(q),\hbox{\ccc
C}^*)$ be the group of characters of the irreducible
representations of the group $A_n(q)$. The group $X$ consists of
all maps
\[
\chi=\chi_b\colon A_n(q)\to\hbox{\ccc C}^*,\hbox{\ }\chi_b(x)=
\varepsilon^{b\left(x\right)},
\]
where $b\in A_n(q)^*$, and the map $A_n(q)^*\to X$,
$b\mapsto\chi_b$, is a group isomorphism. If $\chi=\chi_b$, the
rule $\zeta\cdot\chi=\chi_{\zeta\cdot b}$ defines an action of the
symmetric group $\Sigma_n$ on the group $X$. We denote by
$\overline{\chi}$ the $\Sigma_n$-orbit of the character $\chi\in
X$. Via \emph{transport de structure}, using $t$, we obtain a map
\begin{equation}
t_X\colon\Sigma_n\backslash X\to P_n,\hbox{\ }
\overline{\chi}\mapsto \lambda^{\left(b\right)}.   \label{IV.1.6}
\end{equation}
The images of the maps $t$ and $t_X$ coincide with $P_{\leq
q;n}\subset P_n$. For any $\Sigma_n$-orbit $\overline{b}\in
t^{-1}(\lambda)$, $\lambda\in P_{\leq q;n}$, we choose a
representative $b$ with stabilizer
$\Sigma_\lambda=\Sigma_{\lambda_1}\times\Sigma_{\lambda_2}\times\hdots$,
and denote the corresponding representative $\chi_b$ of
$\overline{\chi_b}\in t_X^{-1}(\lambda)$ via $\chi_{b,\lambda}$.
Moreover, let us set $D_\lambda=A_n(q)\cdot\Sigma_\lambda$. Let
$\mu_i=(\mu_{i1},\mu_{i2},\hdots)$ be a partition of $\lambda_i$,
$i=1,2,\hdots$, and let $[\mu_i]$ be the corresponding irreducible
representation of the group $\Sigma_{\lambda_i}$. The family
$(\gamma_{\mu_1,\mu_2,\hdots}=[\mu_1]\otimes[\mu_2]\otimes\cdots)$
consists of all irreducible representations of $\Sigma_\lambda$.
Let $\hat{\gamma}_{\mu_1,\mu_2,\hdots}$ be the composition of
$\gamma_{\mu_1,\mu_2,\hdots}$ with the canonical projection
$D_\lambda\to\Sigma_\lambda$. Since each character
$\chi_{b,\lambda}$ has stabilizer $\Sigma_\lambda$, it can be
extended to a one-dimensional character of the group $D_\lambda$,
which we denote by the same letter.

\begin{theorem}\label{IV.1.7} (i) The representations
\[
Ind_{D_\lambda}^{B_n\left(q\right)}(\chi_{b,\lambda}\otimes\hat{\gamma}_{\mu_1,\mu_2,\hdots}),
\]
where $\overline{b}\in t^{-1}(\lambda)$, $\lambda\in P_{\leq
q;n}$, are irreducible, pairwise non-isomorphic, and each
irreducible representation of the group $B_n(q)$ has this form;

(ii) the group $B_n(q)$ has exactly
\[
\sum_{\lambda\in P_n}p(\lambda_1)p(\lambda_2)\hdots
m_\lambda(\underbrace{1,\hdots, 1}_{\hbox{$q$ times}},0,\hdots)
\]
irreducible representations.

\end{theorem}

\begin{proof} (i) This is a straightforward use of
\cite[Proposition 25]{[33]}. (ii) The inverse image
$t_X^{-1}(\lambda)$ of any partition $\lambda\in P_{\leq q;n}$
consists of $m_\lambda(\underbrace{1,\hdots, 1}_{\hbox{$q$
times}},0,\hdots)$ elements of the orbit space $\Sigma_n\backslash
A_n(q)^*$, where $m_\lambda(x_1,x_2,\hdots)$ is the monomial
symmetric function, corresponding to $\lambda$.  If we note that
$m_\lambda(1,\hdots, 1,0,\hdots)=0$ when the length of the
partition $\lambda$ exceeds $q$, we obtain that
$|t^{-1}(\lambda)|=m_\lambda(1,\hdots, 1,0,\hdots)$ for any
$\lambda\in P_n$. In order to complete the proof, we use part (i).

\end{proof}

\begin{example}\label{IV.1.8} {\rm
Let us set $d=3$, $n=3$, $\tau=(123)$. Then $q=3$,
$A(3)=(\hbox{\ccc Z}/(3))^3$, $B_3(3)=A(3)\cdot\Sigma_3$.
Moreover, $P_3=\{(3),(2,1),(1^3)\}$, and $P_{\leq 3;3}=P_3$. The
inverse images $t^{-1}(\lambda)$ of the partitions $\lambda\in
P_3$ via the map
\[
t\colon\Sigma_3\backslash A_3(3)^*\to P_3,\hbox{\ }
\overline{b}\mapsto \lambda^{\left(b\right)}
\]
are
\[
t^{-1}((3))=\{\overline{(0,0,0)}, \overline{(1,1,1)},
\overline{(2,2,2)}\},
\]
\[
t^{-1}((2,1))= \{\overline{(0,0,1)}, \overline{(0,0,2)},
\overline{(1,1,0)}, \overline{(1,1,2)}, \overline{(2,2,0)},
\overline{(2,2,1)}\},
\]
and
\[
t^{-1}((1^3))=\{\overline{(0,1,2)}\}.
\]
Further, we obtain
\[
\chi_{\left(0,0,0\right),\left(3\right)}(x)=1,\hbox{\ }
\chi_{\left(1,1,1\right),\left(3\right)}(x)=\varepsilon^{x_1+x_2+x_3},\hbox{\ }
\chi_{\left(2,2,2\right),\left(3\right)}(x)=\varepsilon^{2x_1+2x_2+2x_3},
\]
\[
\chi_{\left(0,0,1\right),\left(2,1\right)}(x)=\varepsilon^{x_3},\hbox{\ }
\chi_{\left(0,0,2\right),\left(2,1\right)}(x)=\varepsilon^{2x_3},\hbox{\ }
\chi_{\left(1,1,0\right),\left(2,1\right)}(x)=\varepsilon^{x_1+x_2},
\]
\[
\chi_{\left(1,1,2\right),\left(2,1\right)}(x)=\varepsilon^{x_1+x_2+2x_3},\hbox{\ }
\chi_{\left(2,2,0\right),\left(2,1\right)}(x)=\varepsilon^{2x_1+2x_2},
\]
\[
\chi_{\left(2,2,1\right),\left(2,1\right)}(x)=\varepsilon^{2x_1+2x_2+x_3},\hbox{\ }
\chi_{\left(0,1,2\right),\left(1^3\right)}(x)=\varepsilon^{x_2+2x_3}.
\]
The irreducible representations of the stabilizer $\Sigma_3$ are
$[(3)]$, $[(1^3)]$, and $[2,1]$ (2-dimensional). Thus, taking into
account that $D_{\left(3\right)}=B_3(3)$, we obtain nine
irreducible representations of the group $B_3(3)$, produced by the
stabilizer $\Sigma_3$ --- six one-dimensional, and three
$2$-dimensional:
\[
\chi_{\left(0,0,0\right),\left(3\right)}\otimes\widehat{\gamma}_{[(3)]}, \hbox{\ }
\chi_{\left(0,0,0\right),\left(3\right)}\otimes\widehat{\gamma}_{[(1^3)]},\hbox{\ }
\chi_{\left(0,0,0\right),\left(3\right)}\otimes\widehat{\gamma}_{[(2,1)]},
\]
\[
\chi_{\left(1,1,1\right),\left(3\right)}\otimes\widehat{\gamma}_{[(3)]},\hbox{\ }
\chi_{\left(1,1,1\right),\left(3\right)}\otimes\widehat{\gamma}_{[(1^3)]},\hbox{\ }
\chi_{\left(1,1,1\right),\left(3\right)}\otimes\widehat{\gamma}_{[(2,1)]},
\]
\[
\chi_{\left(2,2,2\right),\left(3\right)}\otimes\widehat{\gamma}_{[(3)]},\hbox{\ }
\chi_{\left(2,2,2\right),\left(3\right)}\otimes\widehat{\gamma}_{[(1^3)]},\hbox{\ }
\chi_{\left(2,2,2\right),\left(3\right)}\otimes\widehat{\gamma}_{[(2,1)]}.
\]

The stabilizer $\Sigma_2\times\Sigma_1$ has two irreducible
representations: $[(2)]\otimes [(1)]$, and $[(1^2)]\otimes [(1)]$.
Thus, we obtain the following twelve $3$-dimensional irreducible
representations of the group $B_3(3)$:
\[
Ind_{D_{\left(2,1\right)}}^{B_3\left(3\right)}(\chi_{\left(0,0,1\right),\left(2,1\right)}
\otimes\widehat{\gamma}_{[\left(2\right)],[\left(1\right)]}),\hbox{\
}
Ind_{D_{\left(2,1\right)}}^{B_3\left(3\right)}(\chi_{\left(0,0,1\right),\left(2,1\right)}
\widehat{\gamma}_{[\left(1^2\right)],[\left(1\right)]}),
\]
\[
Ind_{D_{\left(2,1\right)}}^{B_3\left(3\right)}(\chi_{\left(0,0,2\right),\left(2,1\right)}
\otimes\widehat{\gamma}_{[\left(2\right)],[\left(1\right)]}),\hbox{\
}
Ind_{D_{\left(2,1\right)}}^{B_3\left(3\right)}(\chi_{\left(0,0,2\right),\left(2,1\right)}
\widehat{\gamma}_{[\left(1^2\right)],[\left(1\right)]}),
\]
\[
Ind_{D_{\left(2,1\right)}}^{B_3\left(3\right)}(\chi_{\left(1,1,0\right),\left(2,1\right)}
\otimes\widehat{\gamma}_{[\left(2\right)],[\left(1\right)]}),\hbox{\
}
Ind_{D_{\left(2,1\right)}}^{B_3\left(3\right)}(\chi_{\left(1,1,0\right),\left(2,1\right)}
\widehat{\gamma}_{[\left(1^2\right)],[\left(1\right)]}),
\]
\[
Ind_{D_{\left(2,1\right)}}^{B_3\left(3\right)}(\chi_{\left(1,1,2\right),\left(2,1\right)}
\otimes\widehat{\gamma}_{[\left(2\right)],[\left(1\right)]}),\hbox{\
}
Ind_{D_{\left(2,1\right)}}^{B_3\left(3\right)}(\chi_{\left(1,1,2\right),\left(2,1\right)}
\widehat{\gamma}_{[\left(1^2\right)],[\left(1\right)]}),
\]
\[
Ind_{D_{\left(2,1\right)}}^{B_3\left(3\right)}(\chi_{\left(2,2,0\right),\left(2,1\right)}
\otimes\widehat{\gamma}_{[\left(2\right)],[\left(1\right)]}),\hbox{\
}
Ind_{D_{\left(2,1\right)}}^{B_3\left(3\right)}(\chi_{\left(2,2,0\right),\left(2,1\right)}
\widehat{\gamma}_{[\left(1^2\right)],[\left(1\right)]}),
\]
\[
Ind_{D_{\left(2,1\right)}}^{B_3\left(3\right)}(\chi_{\left(2,2,1\right),\left(2,1\right)}
\otimes\widehat{\gamma}_{[\left(2\right)],[\left(1\right)]}),\hbox{\
}
Ind_{D_{\left(2,1\right)}}^{B_3\left(3\right)}(\chi_{\left(2,2,1\right),\left(2,1\right)}
\widehat{\gamma}_{[\left(1^2\right)],[\left(1\right)]}).
\]
The trivial stabilizer $\Sigma_1\times\Sigma_1\times\Sigma_1$ has
one irreducible representation --- the unit representation, and,
moreover, $D_{\left(1^3\right)}=A_3(3)$, so we obtain one more
$6$-dimensional irreducible representation of the group $B_3(3)$:
\[
Ind_{A_3\left(3\right)}^{B_3\left(3\right)}(\chi_{\left(0,1,2\right),\left(1^3\right)}).
\]

We note that $6.1^2+3.2^2+12.3^3+1.6^2=3^33!=|B_3(3)|$. }

\end{example}

\end{document}